\documentclass[10pt]{article}
\usepackage{amsmath, amsthm}
\usepackage{amssymb}
\usepackage[hypertex]{hyperref}

\usepackage{graphicx}

\newtheorem{theorem}{Theorem}
\newtheorem{proposition}[theorem]{Proposition}

\newtheorem{sublemma}[theorem]{Sublemma}

\newtheorem{conjecture}[theorem]{Conjecture}

\theoremstyle{definition}
\newtheorem{definition}[theorem]{Definition}

\newtheorem{notation}[theorem]{Notation}

\theoremstyle{remark}
\newtheorem{remark}[theorem]{Remark}

\newcommand{\al}{\alpha}

\renewcommand{\max}{{\rm Max}}
\renewcommand{\min}{{\rm Min}}

\newcommand{\ga}{\mathfrak{a}} 
\newcommand{\gb}{\mathfrak{b}}
\newcommand{\gm}{\mathfrak{m}}

\title{The EGH Conjecture and the Sperner property of complete intersections}
\author{T.\  Harima, Niigata University \\ Department of Mathematics Education, Niigata, 950-2181 Japan
\thanks{Supported by Grant-in-Aid for Scientific Research (C) (15K04812).} 
\and A.\ Wachi, Hokkaido University of Education \\  Department of Mathematics,  
\\ Kushiro, 085-8580 Japan
\thanks{Supported by Grant-in-Aid for Scientific Research (C) (23540179).} 
\and  J.\ Watanabe, Tokai University \\ Department of Mathematics, Hiratsuka, 259-1292 Japan
}
\begin{document}

\maketitle
\date{}

\def\pa{{\partial}}

\begin{abstract} 
Let $A$ be a graded complete intersection over a field and $B$ the monomial 
complete intersection 
with the generators of the same degrees as $A$. The EGH  conjecture says that if 
$I$ is a graded ideal in $A$, then there 
should be an ideal $J$ in $B$ such that $B/J$ and $A/I$ have the same Hilbert function.  
We show that if the EGH conjecture is true, then it can be used to prove that 
every graded complete intersection over any field has the Sperner property. 
\end{abstract}

\section{Introduction}
The Sperner theory for finite posets was started by the paper of Sperner~\cite{sperner} 
and was well established by 1960s.  
The theory served as a model for the thoery of the weak and strong Lefschetz properties of Artinian rings, 
which has been intensively    
studied in the recent years by many authors (\cite{david_cook_1}, \cite{david_cook_2}, 
\cite{HMMNWW},  \cite{MMMNW}, \cite{mcdaniel}, \cite{migliore_miro_roig}, \cite{migliore_nagel}). 
For the detail of the Sperner theory of finite posets we refer the reader to Greene--Kleitman Anderson~\cite{anderson},   
Aigner~\cite{aigner},  Bollob\'{a}s~\cite{bollobas}, Greene-Kleitman~\cite{greene_kleitman}, 
for example. 

A graded Artinian algebra $A$ is said to have the Sperner property if 
$\max\{\mu(\ga)\}$, where $\ga$ runs over all ideals in $A$,  is equal to the dimension of a homogeneous component $A_i$ of 
$A$ of the maximal size. 
It is known that the weak Lefschetz property immediately implies the Sperner property of $A$  
(\cite{Watanabe_2} Corollary, \cite{Watanabe_1} Prop~3.2.(3)).  

In this paper we introduce the notion of the matching property of Gorenstein algebras  
(Definition~\ref{def_matching_property}).    
This is a ring theoretic version of the 
matching property of posets as defined in Aigner~\cite{aigner} VIII \S3. 

Then we will prove that Gorenstein algebras with the 
matching property enjoy the Sperner property 
(Proposition~\ref{matching_implies_sperner}),  
as naturally expected from the definition.    
A reason we are interested in the matching property of Gorenstein algebras is that 
it has a relation to the EGH (Eisenbud-Green-Harris) conjecture (\cite{eisenbud_green_harris}) on complete intersections.    
In Theorem~\ref{main_theorem}  of the present paper we prove that if the EGH conjecture is true it implies that all 
graded complete intersections with a standard grading have the Sperner property.  
As a consequence it implies that all complete intersections for which the EGH conjecture have 
been proved true have the Sperner property. Theorem~\ref{ci_of_product_of_linear_forms} is such a case.   

It would be a good theorem if one could prove that all complete intersections have the Sperner property. 
Being weaker than the WLP, the Sperner property has one advantage over the WLP;  we can expect it does not depend on 
the characteristic of the ground field, while the WLP and SLP 
inescapably depend on the characteristic.   

Throughout this paper a field is used to mean any field (including finite fields) unless otherwise is specified.

\section{The Dilworth number of Artinian graded rings} 

\begin{notation}        
Let $V = \bigoplus _{i=0} ^{\infty} V_i$ be a graded vector space, where $\dim _K V_i < \infty$ for all $i$. 
The map $i \mapsto H(V,i) := \dim _K V_i$ is called the Hilbert function of $V$.  
\end{notation}

\begin{definition}
Let $A = \bigoplus _{i=0} ^{c} A_i$ be an Artinian graded algebra over a field $A_0=K$, and 
$\gm = \bigoplus _{i=1}^c A_i$ the maximal ideal of $A$.  
The {\bf Dilworth number}, denoted  $d(A)$,  is defined to be  
\[d(A):=\max\{\mu(\ga) \mid  \mbox{ ideals }\ga  \subset A \}.\]
We say that $A$ has the {\bf Sperner property} if 
\[d(A)=\max\{\dim _K A_i\mid i=0,1,2,\cdots \}.\] 

\noindent We denote by $\mu(\ga)$ and $\tau(\ga)$ respectively the minimal number of generators and the type of $\ga$.
Namely $\mu(\ga)= \dim _K \ga/\gm\ga$ and $\tau(\ga)=\dim _K (\ga:\gm)/\ga$.  
(These may apply to any local ring $(A, \gm)$ not necessarily Artinian.) 
\end{definition}

For an Artinian Gorenstein local ring $(A, \gm)$  
we introduce two families of ideals:
 \[{\cal F}(A):=\{\ga \mid \mu(\ga)=d(A)\}, \ \  {\cal G}(A):=\{\ga \mid \tau(\ga)=d(A)\}.\] 
If $A$ is graded, we assume that $\ga$ runs over all graded ideals of $A$ (although it does not make much difference).
The following result was proved in Ikeda-Watanabe~\cite{Ikeda_Watanabe_1}. 
\begin{proposition} \label{dilworth_lattice}
${\cal F}(A)$
and ${\cal G}(A)$ are posets with the inclusion as an order, and moreover these are 
lattices with respect to $+$ and $\cap$ as join and meet, and they are isomorphic as lattices via the correspondence:
\[{\cal F}(A) \leftrightarrow {\cal G}(A) \]
\[{\cal F}(A) \ni \ga  \mapsto \gm\ga  \in {\cal G}(A), \]
and 
\[{\cal G}(A) \ni \ga  \mapsto \ga : \gm  \in {\cal F}(A). \]
Furthermore the correspondence 
\[{\cal F}(A)   \to  {\cal G}(A) \]
defined by $\ga \mapsto 0:\ga$ gives an order reversing isomorphism of lattices. 
We assume that $A$ is Gorenstein so we have  $0:(0:\ga) = \ga$, which implies that the same correspondence $\ga \mapsto 0:\ga $ 
gives us the isomorphism in the opposite direction also. 

\end{proposition}


\section{Matching property of monomial complete intersections} 
\label{new_definition}.

\begin{notation}
Let $A=\bigoplus _{i=0}^{\infty} A_i$ be a graded algebra over a field $K=A_0$. 
For a subset $V \subset A_k$, we denote by $A_1\cdot V$ the vector subspace in $A_{k+1}$
spanned by the set 
\[\{xv \mid x \in A_1, v \in V\}.\] 
\end{notation}

\begin{definition} \label{def_matching_property}
Let $A=\bigoplus _{i=0}^c A_i$ be a graded Artinian Gorenstein algebra over a field $K=A_0$. 
We say that $A$ has the {\bf matching property} if we have 
\[\dim _K V \leq \dim _K (A_1 \cdot V) \] 
for any vector space $V \subset A_j$ with $j$ such that $\dim _K A _j \leq \dim _K A_{j+1}$. 
\end{definition}

Over a field of characteristic zero all monomial complete intersections have the WLP, which implies trivially 
the matching property of monomial complete intersections. In Proposition~\ref{prop_matching_property} we show that the 
matching property holds for all monomial complete intersections over a field of any characteristic. 

Let $N$ be a positive integer, and let $P(N)$ be the set of all positive divisors of $N$. Recall that a divisor lattice is 
the set $P(N)$ for some $N$ endowed with a structure of a poset where the partial order is the divisibility.  
It is a poset with rank function, hence we have the rank decomposition $P=\bigsqcup _{i=0}^c P_{i}$. For details see 
\cite{anderson} or \cite{aigner}.   

\begin{proposition} \label{matching_property_of_chain_product}  
Let $P=\bigsqcup _{i=0}^c P_{i}$ be a divisor lattice. 
For a subset  $S \subset P_j$, define the neighbor $N(S)$ of $S$ to be 
\[N(S)=\{y \in P_{j+1} \mid \exists x \in S \mbox{ such that } x < y\}.\]
Then if $|P_j| \leq |P_{j+1}|$, then for any $S \subset P_j$, we have  
\[|S| \leq |N(S)|.\]
\end{proposition}
\begin{proof}  
This follows from the theorem which says $P$ has a symmetric chain decomposition. (See deBruijn et.\ al.\ ~\cite{debruijn_kruyswiek_tengbergen}.)
This also follows from the fact that the monomial complete intersections over a field of characteristic zero have 
the WLP. (See \cite{HMMNWW}.)  
\end{proof}
%


\begin{proposition} \label{prop_matching_property}
Any monomial complete intersection over a field    
has the matching property. 
\end{proposition}

\begin{proof} 
Let $R=K[x_1, x_2, \cdots, x_n]$ be the polynomial ring over a field $K$ 
and write $A=R/\ga=\bigoplus _{i=0}^cA_i$, $\ga=(x_1^{a_1}, x_2^{a_2}, \cdots, x_n^{a_n})$. 
Let $V \subset A_j$ be a vector space. 
We have to show that $\dim _K V \leq \dim _K (A_1 \cdot V)$ if $\dim _K A_{j} \leq \dim _K A_{j+1}$. 
Let $I$ be the ideal of $A$ which is generated by $V$. Then we have  
\[\dim _K V=H(I, j), \ \  \dim _K (A_1\cdot V) = H(I,j+1).\]  
Recall that there exists a monomial ideal $J$ in $A$ such that $A/I$ and $A/J$ have the 
same Hilbert function. 
To see this, let $\phi: R \to A$ be the natural surjection.  Then $\gb:=\phi^{-1}(I)$ contains $\ga$. 
Let ${\rm In}(\gb)$ be the ideal generated by the set of monomials 
that occur as the head term of an element  in $\gb$ w.r.t.\ some monomial order,   
and let $J$ be the image of ${\rm In}(\gb)$ under $\phi$.   
It is well known that   $R/\gb$ and $R/\rm{In}(\gb)$ have the same Hilbert function.  Furthermore 
we have that ${\rm In}(\gb) \supset  \ga$, since monomials are headterms of themselves.  
Since $R/\gb=A/I$ and $R/{\rm In}(\gb)=A/J$, it suffices to prove  $H(J,j) \leq H(J, j+1)$. 
Note that $A$ has the monomial basis $P=\sqcup _{i=0}^c P_i$, which has a structure of the divisor lattice. 
Assume that  $\dim _K A_j \leq \dim _K A_{j+1}$, or equivalently,  $|P_j| \leq |P_{j+1}|$. 
Put $S=J \cap P_j$ and let $J'$ be the ideal generated by $S$.      
Then by Proposition~\ref{matching_property_of_chain_product}, we have 
$\mu(J')=|S| \leq |N(S)|= \mu(\gm J')$, hence $H(J, j) \leq H(J, j+1)$.  
Thus we have proved that $H(I, j) \leq H(I, j+1)$.      
\end{proof}

\section{The Sperner property of Artinian Gorenstein rings}   \label{section_for_sperner_property}

\begin{proposition}  \label{matching_implies_sperner}
Let $A=\bigoplus _{i=0}^c A_i$ be a graded Artinian Gorenstein algebra over a field $A_0=K$ with 
a unimodal Hilbert function. 
Assume that $A$ has the matching property. Then $A$ has the Sperner property.
\end{proposition}

\begin{proof} 
Let $I \subset A$ be an ideal with $\mu(I)=d(A)$, the Dilworth number of $A$. So $I \in {\cal F}(A)$. 
Let $j_0=\min\{j \mid \dim _K A_j > \dim _K A _{j+1}\}$. 
We are going to show that $\mu(I) =\dim _K A_{j_0}$.   
By Watanabe~\cite{Watanabe_1} Lemma~2.4 we may assume that $I$ is graded. 

Let $\al$ be the initial degree of $I$. 
We treat two cases; $\al \geq j_0$ and $\al < j_0$.  

Case 1.   Assume that $\al \geq j_0$.   
Then $\gm ^{j_0} \supset I$ and $\gm ^{j_0+1} \supset \gm I$. 
Put  $J= (0:\gm I)$. Then, bearing in mind that $A$ is Gorenstein, $J \supset 0:\gm ^{j_0+1} = \gm ^{c-j_0}$, and  
thanks to Proposition~\ref{dilworth_lattice},  
we have $J \in {\cal F}(A)$.   
Since the Hilbert function of $A$ is unimodal, we have  $c-j_0 \leq j_0$.  This means 
that $J \supset \gm ^{j_0}$ and the initial degree of $J$ is smaller than or equal to 
$j_0$. If the initial degree of $J$ equals $j_0$, then $J=\gm ^{j_0}$. Otherwise 
we  can use Sublemma~\ref{important_sublemma} repeatedly to obtain an ideal 
$J' \in {\cal F}(A)$ whose initial degree is $j_0$ which contains 
$\gm ^{j_0}$.  This means that $J'=\gm ^{j_0}$ and $\mu(J')= \dim _K A_{j_0}$.  

Case 2. Assume that $\al < j_0$. Then we use again Sublemma~\ref{important_sublemma} repeatedly and 
we obtain an ideal $I' \in {\cal F}(A)$ such that the initial degree of $I'$ is 
at least $j_0$. So this case reduces to Case 1.  
\end{proof}

\begin{sublemma}\label{important_sublemma}  
With the same notation as Proposition~\ref{matching_implies_sperner} suppose that  
$I=\bigoplus _{j \geq \al} I _j$ is a graded ideal of $A$ and  
$\al$ is the initial degree.  Put $I'=\bigoplus _{j \geq \al +1} I _j$.   
Let $j_0=\min\{j \mid \dim _K A_j > \dim _K A _{j+1}\}$ and assume that $\al < j_0$. 
Then we have $\mu(I) \leq \mu(I')$. 
Hence if $I \in {\cal F}(A)$, then $I' \in {\cal F}(A)$. 
\end{sublemma}

\begin{proof}
Note that 
\[\mu(I)= \dim _K I_{\al} + \dim _K I_{\al +1}/(A_1\cdot I_{\al}) + \cdots + \dim _K I_{c}/(A_1\cdot I_{c-1}), \]
and 
\[\mu(I')= \dim _K I_{\al +1} + \dim _K I_{\al + 2}/(A_1\cdot I_{\al +1}) + \cdots + \dim _K I_{c}/(A_1\cdot I_{c-1}). \]
Since $A$ has the matching property, we have  
\begin{eqnarray*}
\mu(I')-\mu(I) & = & \dim _KI_{\al + 1} - \left\{\dim _K I_{\al} + \dim _K I_{\al +1}/(A_1\cdot I_{\al})\right\}\\
                 & = &  \{\dim _K I_{\al +1} - \dim _K I_{\al +1}/(A_1 \cdot I_{\al})\} - \dim _K I_{\al} \\ 
                 & = & \dim _KA_1\cdot I_{\al} - \dim _K I_{\al}  \geq 0.
\end{eqnarray*}
\end{proof}

\section{Main result} \label{sect_main_result}

In this section we show that the EGA Conjecture implies the Sperner property of  complete intersections.  
Recall that the EGH Conjecture is as follows: 
\begin{conjecture}[EGH Conjecture~\cite{eisenbud_green_harris}] \label{egh_conjecture} 
Let  $R=K[x_1, x_2, \cdots, x_n]$ be the polynomial ring over a field $K$.  
If $I$ is a graded ideal in $R$ containing a regular sequence $f_1, f_2, \cdots, f_n$ of 
degrees $a_1, a_2, \cdots, a_n$ respectively, then $I$ has the same Hilbert function as an 
ideal containing $(x_1^{a_1}, x_2^{a_2}, \cdots, x_n^{a_n})$.   
\end{conjecture}

\begin{theorem}  \label{main_theorem} 
Suppose that Eisenbud-Green-Harris Conjecture is true for a graded complete intersection $A$ over a field $K$.  
Then $A$ has the Sperner property.  
\end{theorem}

\begin{proof} 
Let $A=\bigoplus _{i=0}^c A_i$ be a graded complete intersection defined by the ideal generated by homogeneous elements 
\[f_1, f_2, \cdots, f_n \in R=K[x_1, x_2, \cdots, x_n]\]
of degrees $a_1, a_2, \cdots, a_n$ respectively.  
Let $j$ be an integer such that $\dim _K A_j \leq \dim _K A_{j+1}$ and 
let $V \subset A_j$ be any vector subspace. 
By Proposition~\ref{matching_implies_sperner} it is enough to show that 
\[\dim _K V \leq \dim _K (A_1\cdot V).\] 
Let $I=VA$, namely,  $I$ is the ideal in  $A$ generated by $V$. 
Then 
\[\dim _K V= \mu(I), \;  \; \dim _K (A_1\cdot V) = \mu(\gm I).\]
By Proposition~\ref{hilbert_series_vs_numgens}, which we prove below, we have   
\begin{equation}\label{eq_1}
\mu(I) = \dim _K V= H(A,j)-H(A/I, j),
\end{equation}
and    
\begin{equation}\label{eq_2} 
\mu(\gm I) = \dim _K (A_1\cdot V) = H(A,j+1)-H(A/I, j+1). 
\end{equation}

Let $\ga \subset R$ be a graded ideal such that 
$A/I=R/\ga$, where $\ga \supset (f_1, f_2, \cdots, f_n)$.  
Now we invoke the EGH Conjecture for $A$, which is to assume the following: 
There is an ideal $\gb$ containing $(x_1 ^{a_1}, \cdots, x_n^{a_n})$ such that $R/\gb$ has the same Hilbert 
function as that of $R/\ga$. 
Put $B=R/(x_1^{a_1}, \cdots, x_n^{a_n})$ and  $J=\gb/(x_1^{a_1}, \cdots, x_n^{a_n})$.  We regard $J$ as an ideal in $B$. 
By definition $I$ has initial degree $j$, so $j$ is the initial degree of $J$.  
Let $J'$ be the ideal of $B$ generated by the elements of $J_j$.

We want to show that 
$\mu(\gm I) - \mu(I) \geq 0$. 
By (\ref{eq_1}) and (\ref{eq_2})
\begin{equation*}
\mu(\gm I)-\mu(I)  =   \{H(A, j+1) - H(A/I, j+1)\} -\{H(A, j) - H(A/I, j)\}. 
\end{equation*}
Since $A$ and $B$ have the same Hilbert function and so do $A/I$ and $B/J$, this is equal to  
\begin{eqnarray}   \label{eq_3}   
\mu(\gm I)-\mu(I)  =  \{H(B, j+1) - H(B/J, j+1)\} -\{H(B, j) - H(B/J, j)\}.  
\end{eqnarray}
As well as for $I$, we have  
\begin{equation} \label{eq_10}
\mu(\gm J')-\mu(J ')   =   \{H(B, j+1) - H(B/J ', j+1)\} -\{H(B, j) - H(B/J ', j)\},  
\end{equation} 
where we used the same $\gm$ for the maximal ideal of $B$. 
Note that $\mu(J')=\dim _K J'_j$ and $\mu(\gm J') = \dim _K (A_1 \cdot J')$. 
Since $B$ has the matching property (Proposition~\ref{prop_matching_property}), 
we have  $\mu(\gm J') - \mu(J') \geq 0$, which implies that 
the RHS of (\ref{eq_10}) is $ \geq 0$. 
Since $J ' \subset J$,  we have 
\[H(B/J ', j+1) - H(B/J, j+1) \geq 0.\]      
Add this number to the RHS of  (\ref{eq_10}). Then  we obtain 
\[\{H(B, j+1) - H(B/J, j+1)\} -\{H(B, j) - H(B/J ', j)\}  \geq 0.\]
Since $J '_j = J _j$, we have 
\[\{H(B, j+1) - H(B/J, j+1)\} -\{H(B, j) - H(B/J, j)\}  \geq 0,\]
which implies that the LHS of (\ref{eq_3}) is $\geq 0$. 
Thus we have proved that 
\[\dim _KV \leq \dim _K(A_1\cdot V),\]
as desired. 
\end{proof}
\begin{proposition} \label{hilbert_series_vs_numgens}  
Let  $R=K[x_1, \cdots, x_n]$  be the polynomial ring and $\ga$ a graded ideal of finite colength.
Put $A=R/\ga$.  Let $I$ be an ideal of $A$ generated by a subset in $A_{\al}$.  

Then 
\[\mu_A(\gm ^i I)=H(A, \al+i) - H(A/I, \al+i).\]
\end{proposition}
\begin{proof}
\begin{eqnarray*}
\mu _A(\gm ^{i} I) & = &  \dim _K \gm ^{i}I/\gm ^{i+1}I = \dim _K I _{\al+i}  \\
                & = & H(I, \al+i)  \\                
                & = & H(A, \al+i) - H(A/I, \al+i)                 
\end{eqnarray*}
\end{proof}

\begin{theorem}\label{ci_of_product_of_linear_forms}
Let $R=K[x_1, x_2, \cdots, x_n]$ be the polynomial ring. 
Suppose that $I$ is a complete intersection ideal  
  $I=(L_1, L_2, \cdots,  L_n)$, 
where each $L_j$  is a product of linear forms. Then $A:=R/I$ has the Sperner property.
\end{theorem}

\begin{proof}
By \cite{abed_abedelfatah} Corollary~3.3, the EGH Conjecture is true for $A$. So Theorem~\ref{main_theorem}  applies. 
\end{proof}

\begin{remark}  
\rm{In the above theorem we may require only that $(n-2)$ of the generators 
are products of linear forms and  the other two are arbitrary homogeneous elements. 
Indeed,  to prove the theorem we may assume that $K$ is algebraically closed. 
The proof is carried out by induction on $n$. If $n=2$, then any homogeneous form factors into a 
product of linear forms over $K$.  
}
\end{remark}


\end{document}